\documentclass[hidelinks]{article}         										
\usepackage[english,strings]{babel}     
\usepackage{amsmath}   
\usepackage[utf8]{inputenc}    
\usepackage{longtable}         
\usepackage{exscale}           
\usepackage[final]{graphicx}   
\usepackage[sort]{cite}        
\usepackage{array}             
\usepackage{fancyhdr}          
\usepackage[a4paper]{geometry} 
\usepackage{xspace}            
\usepackage{tikz-cd}              
\tikzcdset{scale cd/.style={every label/.append style={scale=#1}, 
		cells={nodes={scale=#1}}}}
\usepackage{ifdraft}           
\usepackage{nchairx} 
\usepackage[sectionbib         
]{chapterbib}       
\usepackage{appendix}
\usepackage[expansion=false    
]{microtype}        
\usepackage[nottoc]{tocbibind} 
\usepackage[backref=page,      
final=true,         
pdfpagelabels       
]{hyperref}         
\usepackage{color}
\usepackage{dsfont}
\usepackage{tikz}
\usepackage{CJKutf8}
\usepackage{bbold} 
\usepackage{yfonts}
\usepackage{amssymb}
\usepackage{epigraph}

\AtEndDocument{\bigskip{\footnotesize
		\noindent\textsc{Seoul National University, Department of Mathematical Sciences, 	Research institute in Mathematics, Gwanak-Gu, 
			Seoul 08826, South Korea} \par  
		\noindent    \textit{E-mail address}: \texttt{\href{mailto:kevin.ruck@snu.ac.kr}{kevin.ruck@snu.ac.kr}} \par
}}

\begin{document}
	\title{A Floer Theoretic Approach to Energy Eigenstates on one Dimensional Configuration Spaces}
	\author{Kevin Ruck} 
	\date{}
	
	\maketitle
	
	\begin{abstract}
		In this article we consider two classical problems in Quantum Mechanics, namely the 'particle on a ring' and the 'particle in a box', from the viewpoint of symplectic topology. Interpreting the solutions of the  corresponding time independent Schr\"odinger equation as orbits in a suitably chosen time dependent Hamiltonian system allows us to investigate them using Floer theory. More precisely we extend the definition of Rabinowitz Floer homology to non-autonomous Hamiltonians on $\mathbb{R}^{2n}$ with its standard symplectic structure and show that compactness of the moduli space of J-holomorphic curves still holds. With this homology we are then able to prove existence results for energy $E$ eigenstates on the 'ring' or in the 'box' for a big range of exterior potentials.
	\end{abstract}
	\medskip
	
	\noindent\textbf{Keywords: }  Quantum Mechanics  $\cdot$ Floer Homology $\cdot$ non-autonomous Hamiltonian
	
	\medskip
	
	\noindent\textbf{Mathematics Subject Classification:}  81Q35 $\cdot$ 	53D40 $\cdot$ 37J65
	
	\section{Introduction}
	Symplectic topology proved over the last decades its enormous usefulness in the study of classical mechanical systems, especially when it comes to existence results for specific types of solutions (e.g. periodic orbits, consecutive collision orbits, etc.). On the side of Quantum Mechanics the pioneering work of Maurice de Gosson (\cite{deGosson2007a}, \cite{deGosson2009a}, etc.) and Leonid Polterovich (\cite{polterovich2012a}, \cite{polterovich2014a}, etc.) showed that symplectic topology also has its use in the study of quantum phenomenon. In this paper we want to contribute to building a bridge between those two areas, in particular we want to explore how we can use Floer theory to prove existence results for specific solutions of the Schr\"odinger equation. In classical mechanics the fundamental solutions of the Hamiltonian equation are the periodic orbits, which are according to Poincar\'e the 'only gateway into the otherwise impenetrable domain of nonlinear dynamics'. The most fundamental solutions of the Schr\"odinger equation in quantum mechanics are in this sense the energy eigenstates.
	
	A well established tool in symplectic geometry to study the existence of periodic orbits in a Hamiltonian system is Rabinowitz Floer homology. So following the analogy made above, our goal in this paper is to explore if we can use (a variation of) Rabinowitz Floer homology to prove an existence statement for energy eigenstates in a given quantum mechanical system. 
	
	The two quantum systems we chose for this paper are called 'particle on a ring' and 'particle in a box'. As their names suggest they consist of a quantum particle whose configuration space is a circle or a closed interval potentially under the influence of an exterior field (for more details see Section~\ref{QM}). The reason why we chose these two systems is that they exhibit all the fundamental features of a quantum system while still being of rather low complexity. 
	
	So given a two dimensional plane with an exterior field which is described by a potential $\tilde{V}:\mathbb{R}^{2}\to \mathbb{R}$, i.e. it only depends on the position of the particle. In this plane we want to prepare a quantum state that is confined to a ring or a box and is an eigenstate for energy $E$.\vspace{2mm}
	
	\noindent\textbf{Question:} \textit{For a given potential $\tilde{V}$ and any given energy $E$ can we always find a radius $R$ such that we can realize an energy eigenstate with energy $E$ on a ring of radius $R$ (alt. a box of length $R$)?} \vspace{2mm}

	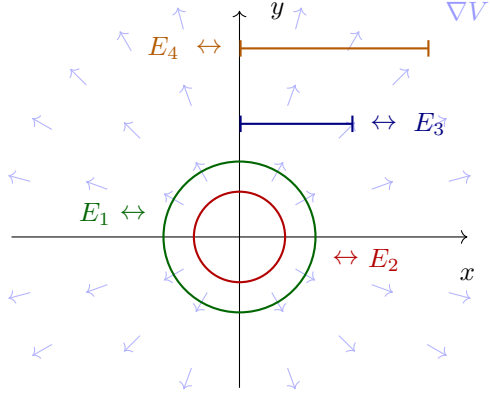
\begin{figure}
		\begin{center}
			\begin{tikzpicture}
				\draw[->] (-3,0)--(3,0);
				\draw[->] (0,-2)--(0,3);
				\draw (3,-0.5) node {$x$};
				\draw (0.5,3) node {$y$};
				\foreach \i/\j in {1/30,1/60,1/120,1/150,1/210,1/240,1/300,1/330,2/20,2/45,2/70,2/110,2/135,2/160,2/200,2/225,2/250,2/290,2/315,2/340,3/15,3/30,3/45,3/60,3/75,3/105,3/120,3/135,3/150,3/165,3/195,3/210,3/330,3/345} {
					\draw[blue!30!white,->] (\j:\i-0.15)--(\j:\i+0.15);
				}
				\draw[blue!40!white] (3,3) node {$\nabla V$};
				\draw[green!40!black,thick] (0,0) circle[radius=1cm];
				\draw[green!40!black] (-1.4,0.3) node {$\leftrightarrow$};
				\draw[green!40!black] (-1.9,0.3) node {$E_1$};
				\draw[orange!70!black,thick,|-|] (0,2.5)--(2.5,2.5);
				\draw[red!70!black] (1.4,-0.3) node {$\leftrightarrow$};
				\draw[red!70!black] (1.9,-0.3) node {$E_2$};
				\draw[blue!50!black,thick,|-|] (0,1.5) -- (1.5,1.5);
				\draw[blue!50!black] (1.9,1.5) node {$\leftrightarrow$};
				\draw[blue!50!black] (2.5,1.5) node {$E_3$};
				\draw[red!70!black,thick] (0, 0) circle[x radius=0.6, y radius=0.6];
				\draw[orange!70!black] (-0.4,2.5) node {$\leftrightarrow$};
				\draw[orange!70!black] (-1,2.5) node {$E_4$};
			\end{tikzpicture}
		\end{center}
		\caption{A sketch of two rings and two boxes in an exterior field $\nabla V$ such that there exists an energy $E_i$ eigenstate on them.}
		\label{setup}
	\end{figure}

	\noindent To solve this question we use the following approach: When confined to a circle of radius $R$, the Hamiltonian operator for the above described setting is given by the unbounded operator
	\begin{align*}
		\hat{H}&:\ L^2(S^1,\mathbb{C}) \supset\mathcal{U} \to L^2(S^1, \mathbb{C})\\
		\psi &\mapsto -\frac{1}{R^2}\frac{\mathrm{d}^2}{\mathrm{d}\varphi^2}\psi + V\cdot \psi,
	\end{align*}  
	where $\mathcal{U}$ is an appropriately chosen dense subspace and $V$ is the exterior potential on the plane restricted to the circle. Note that we chose the units in such a way that $\frac{\hbar^2}{2m}=1$. For $\psi$ to be an energy eigenstate it therefore has to satisfy the 2nd order ODE
	\begin{equation}
		\hat{H} \psi = -\frac{1}{R^2}\frac{\mathrm{d}^2}{\mathrm{d}\varphi^2}\psi + V\cdot \psi= E\cdot \psi. 
		\label{eigenstate}
	\end{equation}
	Now consider the Hamiltonian function 
	\begin{equation}
		H:\mathbb{R}\times\mathbb{R}^4\to \mathbb{R}; \qquad (t,q,p)\mapsto\frac{\|p\|^2}{2}+\left(E-V\left(\frac{t}{R}\right)\right)\frac{\|q\|^2}{2}-c,
		\label{mainH}
	\end{equation}
	where $V\left(\frac{t}{R}\right)$ extends to $\mathbb{R}$ as a periodic function and $c>0$ is a constant. The corresponding equation of motion is given by
	\begin{align*}
		\frac{\D}{\D t} q(t)&=p(t)\\
		\frac{\D}{\D t} p(t)&=-\left(E-V\left(\frac{t}{R}\right)\right)q(t).
	\end{align*}
	Assume  there exists a periodic solution $(Q(t),P(t))$ with period $R$. Then we define the function
	\[
		\psi_E(\varphi):= Q_1(R\varphi)+iQ_2(R\varphi) \in L^2\left(S^1,\mathbb{C}\right).
	\]
	This function satisfies
	\begin{align*}
		-\frac{1}{R^2}\frac{\mathrm{d}^2}{\mathrm{d}\varphi^2}\psi_E + V\cdot \psi_E&= \frac{\mathrm{d}^2}{\mathrm{d}t^2}\left(Q_1(t)+iQ_2(t)\right)+V\left(\frac{t}{R}\right)\cdot\left(Q_1(t)+iQ_2(t)\right)\\
		&=E\left(Q_1(t)+iQ_2(t)\right)\\
		&=E\psi_E(\varphi),
	\end{align*}
	where we used the substitution $\varphi=\frac{t}{R}$ for the first and third equality and the fact that $Q(t)$ is a solution to the above equation of motion for the second equality. Therefore $\psi_E$ is an energy $E$ eigenstate. This approach allows us to study the existence of energy eigenstates on a ring of radius $R$ by looking for $R$-periodic orbits for the Hamiltonian function \eqref{mainH}. 
	
	In symplectic topology there exists a variety of different homologies that are able to detect periodic orbits in a given Hamiltonian system. The one suitable for our goal is call Rabinowitz Floer homology. However, the original definition of this homology requires a well-defined energy hypersurface, i.e. an autonomous Hamiltonian function. Hence, our strategy will be to extend the definition of Rabinowitz Floer homology to non-autonomous Hamiltonians on $\mathbb{R}^{2n}$, which will enable us to prove the following existence result:
	\begin{theorem}
		Given a radially invariant potential $\tilde{V}:\mathbb{R}^2\setminus\{0\}\to \mathbb{R}$. Then for every energy $E$ with 
		\begin{align*}
			E>\max \tilde{V}
		\end{align*}
		there exists a radius $r_E\in \mathbb{R}^+$ such that on the circle $\del B_{r_E}(0)$ there exists an energy eigenstate for energy $E$ of the Hamiltonian operator $\hat{H}\psi=-\frac{1}{r_E^2}\frac{\mathrm{d}^2}{\mathrm{d}\varphi^2}\psi + V\cdot \psi$, where $V$ is the restriction of $\tilde{V}$ to $\del B_{r_E}(0)$
		\label{thmclosed}
	\end{theorem}
	
	For the 'particle in the box' we can use a similar strategy: Consider a Hamiltonian trajectory $(Q(t),P(t))$ of \eqref{mainH} such that
	\begin{align*}
		(Q(t),P(t))\in \{0\}\times \mathbb{R}^2
	\end{align*}
	for $t=0$ and a later time $t=T$. Then again 
	\begin{align*}
		\psi(\varphi):= Q_1(\varphi)+iQ_2(\varphi)
	\end{align*}
	is an energy eigenstate to energy $E$ in the 'box' of length $T$. Since $\{0\}\times \mathbb{R}^2$ is a Lagrangian submanifold we can use a modified version of Lagrangian Rabinowitz Floer homology to prove the following theorem:
	\begin{theorem}
		Given a radially invariant potential $\tilde{V}:\mathbb{R}^2\setminus\{0\}\to \mathbb{R}$. Then for every energy $E$ with 
		\begin{align*}
			E>\max \tilde{V}
		\end{align*}
		there exists a length $l_E\in\mathbb{R}^+$ such that on the straight line from $(0,l_E)$ to $(l_E,l_E)$ in $\mathbb{R}^2$ (c.f.  Figure~2 in Chapter~\ref{openstring}), there exists an energy eigenstate for energy $E$ of the Hamiltonian operator $\hat{H}\psi=-\frac{\mathrm{d}^2}{\mathrm{d}s^2}\psi + V\cdot \psi$, where $V$ is the restriction of $\tilde{V}$ to the straight line.
		\label{thmopen}
	\end{theorem}
	
	\begin{remark}
		In the above theorem for clarity we chose to give a very explicit statement about the type of 'box' on which one can find an energy $E$ eigenstate. From the proof of this theorem in Chapter~\ref{openstring} one can see that we are free to choose any type of orientation of the line in $\mathbb{R}^2$ by adapting the positioning of the reference line $\hat{\Gamma}$.
	\end{remark}
	The outline of the paper is a follows: We start in Section~2 with a short introduction to the basic concepts in quantum mechanics and how they relate to their classical counterpart in Hamiltonian mechanics. 
	
	Section~3 is about extending the definition of Rabinowitz Floer homology to non-autonomous Hamiltonians. This will be done in two parts: First we consider the case of periodic orbits as generators of the homology as in the original definition in \cite{cieliebak2009a} and afterwards we discuss the case of chords starting and ending in a given Lagrangian submanifold as generators as introduced in \cite{merry2014a}. Originally the well-definedness of Rabinowitz Floer homology relies on the existence of a contact form on a given energy hypersurface, however, for a time dependent Hamiltonian there does not even exist an energy hypersurface. So instead of considering Hamiltonian trajectories that lie on a fixed energy hypersurface we look at orbits whose average energy is equal to a fixed value. Restricting our considerations to only $\mathbb{R}^{2n}$ with its standard symplectic structure as the symplectic manifold with 'nice enough' Hamiltonian functions on it allows us then to prove the compactness of the moduli space of gradient flow lines also in the absence of a contact energy hypersurface.
	
	In Section~4 we utilize the in Section~3 defined Rabinowitz Floer homology for non-autonomous Hamiltonians to prove Theorem~\ref{thmclosed} and Theorem~\ref{thmopen}). The key idea of these proofs is that in one dimensions we can identify the eigenstates for the Hamiltonian operator with the position component of the solution to the Hamiltonian equation of motion for a suitably chosen Hamiltonian function, as explained above. After showing that the RFH vanishes in these settings, the two theorems follow via the usually proof by contradiction: Assume there is no non-constant solution then RFH has to be equal to the singular homology of a closed (non empty) submanifold, which contradicts the vanishing of RFH.
	
	\subsection*{Acknowledgments}
	The author would like to thank Jungsoo Kang and Otto van Koert for helpful discussions. This work was supported by the National Research Foundation of Korea under grants RS-2023-00211186, (MSIT) RS-2023-NR076656 and  NRF-2020R1A5A1016126.
	
	\section{A short Introduction to Quantum Mechanics}
	\label{QM}
	For the readers that are not familiar with quantum mechanics we want to give a short and basic introduction of the fundamental concepts.
	
	First, we have to introduce the notion of \textit{observable}: An observable is something that we can measure (observe) in a given physical system, like energy, position, or angular momentum. In classical mechanics we usually have $C^\infty\left(T^*Q\right)$ as the space of observables, containing the Hamiltonian function $H$, the position coordinates $q_i$, momentum coordinates $p_i$, etc. The value of the observables depend on the \textit{state} of the physical system. For a point mass in a Hamiltonian system that follows a trajectory $\gamma$ the state at a given time $t_0$ is simply $\gamma(t_0)$, i.e. its state at any given time can be completely described by the position and momentum observables. The process of 'measuring' a given observable $f\in C^\infty\left(T^*Q\right)$ when the physical system is in the state $\gamma(t_0)=z_0$ is described by the evaluation map
	\[
	\epsilon_{z_0}: C^\infty\left(T^*Q\right) \to \mathbb{R}; \quad f\mapsto f(z_0).
	\]
	
	In quantum mechanics the observables are linear operators on a separable Hilbert space. Note that in physics textbooks the precise mathematical description of these operators is usually avoided and instead they treat these objector more intuitively. Giving a mathematically precise definition usually requires the introduction of densely defined self adjoint operators and a lot of additional back ground in functional analysis. So to keep this introduction short and basic we will stick more to the physics textbook approach. 
	
	To find the right analogue for an observable that we know from classical mechanics (like energy, position, etc.) in the quantum mechanical setting one uses a process called \textit{quantization}. Again to keep this section short and basic we will describe the process of quantization in a very rudimental way: The idea is that quantum and classical mechanics are quite different, but still there should exist some kind of connection between them, since classic mechanics is just a limit case of quantum mechanics for large scales. One way physicists found to relate those two worlds is called \textit{canonical quantization}. Given a classical observable like the Hamiltonian function 
	\[
	H:\mathbb{R}^2\to\mathbb{R}; \quad (q,p)\mapsto\frac{p^2}{2}+V(q),
	\]
	where the position space is for simplicity just $\mathbb{R}$, one turns it into the corresponding quantum mechanical energy observable $\hat{H} \in \mathscr{L}\left(L^2(\mathbb{R},\mathbb{C})\right)$ by substituting the momentum coordinate $p$ by the differential operator
	\[
	-i\hbar\frac{\D}{\D x},
	\] 
	where $\hbar$ is a positive constant, and the position dependent function $V$ by the linear operator
	\[
	\hat{V}: L^2(\mathbb{R},\mathbb{C}) \to L^2(\mathbb{R},\mathbb{C}); \quad \psi \mapsto V\cdot \psi.
	\]
	Note that both of these fundamental operators are in general not even well-defined on all of $L^2(\mathbb{R},\mathbb{C})$, but we chose to ignore this for now. Performing these substitutions for the above Hamiltonian $H$ yields the \textit{Hamiltonian operator}
	\[
	\hat{H}:= -\frac{\hbar^2}{2}\frac{\D^2}{\D x^2}+ \hat{V}.
	\]
	The state of a quantum mechanical system (for position space $\mathbb{R}$) is a function $\psi\in L^2(\mathbb{R},\mathbb{C})$ with $\|\psi\|=1$ called the \textit{wave function}. Measurements in quantum mechanics are fundamentally different to classical mechanics in that even theoretically perfect measurements don't need to give the same measurements every time. Hence, we can only talk about the expectation value of a given observable for a given state. This is given by
	\[
	\langle\psi,\hat{A}\psi\rangle\,
	\]
	where this inner product is the standard scalar product of $L^2(\mathbb{R},\mathbb{C})$. The statistical variance of the measurements can be calculated via
	\begin{equation}
		\langle\psi,\hat{A}^2\psi\rangle-\langle\psi,\hat{A}\psi\rangle^2.
		\label{var}
	\end{equation}
	
	One very important concept in quantum mechanics are the eigenstates for a given observable $\hat{A}$, i.e. a state $\phi \in L^2(\mathbb{R},\mathbb{C})$ such that
	\[
	\hat{A}\phi = \lambda \phi
	\]
	for $\lambda\in\mathbb{R}$. The reason why these states are so special is the fact that when we measure the observable $\hat{A}$ in a system that is in the eigenstate $\phi$ the result is always $\lambda$ (c.f. equation \eqref{var}). Further, given an arbitrary state $\psi$ we can determine the probability that the measurement of $\hat{A}$ for the system in state $\psi$ is equal to $\lambda$ by calculating
	\[
	|\langle \phi,\psi\rangle|^2.
	\]
	The most visible display of the importance of eigenstates is the relation between energy eigenstates of electrons in a given atom or molecule and its corresponding light emission spectrum: For simplicity consider a hydrogen atom, which consists only of one proton and one electron. The classical Hamiltonian for the electron is the ordinary Kepler Hamiltonian
	\[
	H(q,p)=\frac{\|p\|^2}{2}-\frac{1}{\|q\|},
	\] 
	where we as usually set all physical constants equal to $1$. After quantization the corresponding Hamiltonian operator is given by
	\[
	\left(\hat{H}\psi\right)(x):=-\frac{1}{2}\Delta\psi(x) -\frac{1}{\|x\|}\psi(x).
	\] 
	When 'exciting' the electron in the hydrogen atom, i.e. giving it more energy, the electron will emit the additional energy again in form of a photon. But the energy of this photon can only be equal to the difference of two energy eigenvalues of $\hat{H}$. Note that this model of the hydrogen atom is only an approximation and the real emission spectrum shows several small deviations from the energy eigenvalues of $\hat{H}$. Nevertheless it illustrates the impotence of energy eigenstates and should serve as a motivation to study energy eigenstates in more details.
	
	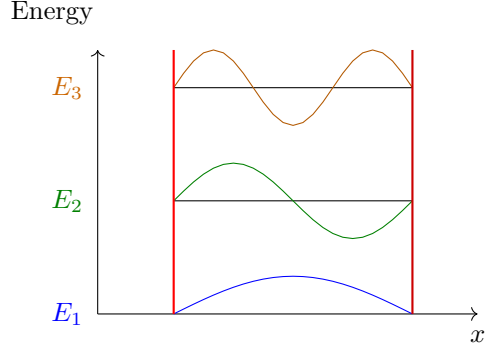
\begin{figure}
		\begin{center}
			\begin{tikzpicture}[domain=0:3.1415]
				\draw[->] (-1,0)--(4,0);
				\draw (4,-0.3) node {$x$};
				\draw[->] (-1,0)--(-1,3.5);
				\draw[thin] (0,1.5)--(3.1415,1.5);
				\draw[thin] (0,3)--(3.1415,3);
				\draw (-1.6,4) node {Energy};
				\draw[blue]   plot (\x,{0.5*sin((\x r))}); 
				\draw[green!50!black]   plot (\x,{0.5*sin(2*(\x r))+1.5}); 
				\draw[orange!70!black]   plot (\x,{0.5*sin(3*(\x r))+3}); 
				\draw[red,thick] (0,0)--(0,3.5);
				\draw[red!80!black,thick] (3.1415,0)--(3.1415,3.5);
				\draw[blue] (-1.4,0) node {$E_1$};
				\draw[green!50!black] (-1.4,1.5) node {$E_2$};
				\draw[orange!80!black] (-1.4,3) node {$E_3$};
			\end{tikzpicture}
		\end{center}
		\caption{A sketch of the different energy eigenstates for the 'free particle in the box'}
	\end{figure}
	
	The easiest system one can consider to study its energy eigenstates is the famous 'free particle in a box'. In Physics this system is usually described as a free particle in one dimension trapped in a finite area between two infinitely high potential walls. Inside this finite area the Hamiltonian operator has the simple form
	\[
		 -\frac{\hbar^2}{2}\frac{\D^2}{\D x^2}.
	\] 
	So for a box that starts at $x=0$ and ends at $x=\pi$ the energy eigenstates are given by
	\[
		\psi_n(x):= A\sin(n\pi x)
	\]
	for $x$ inside the box and zero outside and $A$ a constant in $\mathbb{C}$. Note that the 'infinite potential' at the boundary of the box forces the wave function to be zero from this point on. The corresponding energy eigenvalues are $E_n=\frac{n^2\pi^2 \hbar^2}{2}$.

	The 'free particle on a ring' is a variation of the above setting. Instead of confining the free particle to an interval one considers $S^1$ as configuration space. The boundary condition in this case requires the wave functions to be periodic with respect to the chosen ring. 
	
	For the setting considered in this paper we added additional complexity to these simple quantum systems by considering a position dependent potential in $\mathbb{R}^2$ and embedding the box or the ring respectively into this plane. Hence, the particle can no longer be described as a free particle on the confined space, which makes finding the energy eigenstates much more challenging. 
	
	\section{Time dependent Rabinowitz Floer Homology}
	So far Rabinowitz Floer homology (RFH) was only considered for time independent Hamiltonians, which is not surprising since its definition traditionally relies on the existence of an energy hypersurface that is of contact type. However, because the setting in which we want to use a time dependent version of RFH is very simple (just $\mathbb{R}^4$ with a time dependent harmonic oscillator (c.f. \eqref{mainH}), we will be able to prove the usual well-definedness result without the existence of a contact hypersurface.
	\subsection{Periodic RFH}
	\label{periodiccase}
	For a given time dependent Hamiltonian $H:S^1\times\mathbb{R}^{2n} \to \mathbb{R}$ on the symplectic manifold $\left(\mathbb{R}^{2n},\omega\right)$ we can define the Rabinowitz action functional analogue to the time independent case:
	\begin{align*}
		\mathcal{A}_{H_t}:= \int\limits_{S^1} x^*\lambda - \tau\int\limits_{S^1} H_t(x(t))\mathrm{d}t,
	\end{align*}
	where $\lambda$ is a one form such that $\mathrm{d}\lambda=\omega$. The corresponding critical point equation then reads
	\begin{align*}
		\frac{\mathrm{d}}{\mathrm{d}t} x(t)=& \tau X_{H_t}(x(t))\\
		\int\limits_{S^1}H_t(x(t))\mathrm{d}t=&0
	\end{align*}
	This means for a critical point $(x,\tau)$ the rescaled loop $x\left(\frac{t}{\tau}\right)$ is a Hamiltonian trajectory with period $\tau$ of the Hamiltonian $H_{\frac{t}{\tau}}$. In a given mechanical system it is usually very uncommon that we can choose the period of the time dependent Hamiltonian to our liking. Nevertheless, for the question we set out to answer in this paper, there is still a lot  we can achieve regardless of this draw back. If one would instead consider the action functional
	\begin{align*}
		\tilde{\mathcal{A}}_{H_t}:= \int\limits_{S^1} x^*\lambda - \tau\int\limits_{S^1} H_{\tau t}(x(t))\mathrm{d}t,
	\end{align*}
	the loop $x\left(\frac{t}{\tau}\right)$ would be a trajectory of $H_t$, i.e. the period of the Hamiltonian does not vary for different critical points. However, the corresponding gradient flow equation for this functional
	\begin{align*}
		\del_su(s,t)=\del_t u(s,t) -\eta(s)H_{\eta(s)t}(u(s,t))
	\end{align*}
	will not necessary lead to the usual Floer cylinders, which poses serious questions on the well-definedness of such a Floer homology. Therefore we chose to work with the former functional. 
	
	Since in the original definition of RFH the compactness of the moduli space of Floer cylinders relies strongly on the existence of a contact hypersurface (c.f. \cite{cieliebak2009a}) the main task of this chapter is finding a proof that (in a less general setting) works without the contact condition. First, we make the following observation: Let $(x,\tau)$ be a critical point, then we calculate 
	\begin{align*}
		\mathcal{A}_{H_t}= \int\limits_{S^1} x^*\lambda - \tau\int\limits_{S^1} H_t(x(t))\mathrm{d}t= \int\limits_{S^1}\lambda_{x(t)}\left(\tau X_{H_t}(x(t))\right)\mathrm{d}t.
	\end{align*} 
	Further, note that
	\begin{align*}
		-\D H_t=\iota_{X_{H_t}}\omega=\iota_{X_{H_t}}\D\lambda=-\D\iota_{X_{H_t}}\lambda+\Lie_{X_{H_t}}\lambda.
	\end{align*}
	Hence, if we can find a one form $\lambda$ with $\D\lambda=\omega$ that is invariant under the Hamiltonian flow $\Phi^t_{H_t}$, then we can deduce that
	\[
	-\D H_t=-\D \left(\iota_{X_{H_t}}\lambda\right)
	\]
	which implies
	\[
	\iota_{X_{H_t}}\lambda=H_t+c_0(t)
	\]
	for a function $c_0:S^1\to \mathbb{R}$ that only depends on time. Overall we get for a critical point $(x,\tau)$
	\begin{align*}
		\mathcal{A}_{H_t}= \int\limits_{S^1}\lambda_{x(t)}\left(\tau X_{H_t}(x(t))\right)\mathrm{d}t= \int\limits_{S^1}\tau\cdot\left(\ H_t(x(t))+c_0(t)\right)\mathrm{d}t= 0+\tau\int\limits_{S^1}c_0(t)\mathrm{d}t=:\tau\cdot\hat{c}_0,
	\end{align*}
	where $\hat{c}_0$ is a constant. So if this constant $\hat{c}_0\ne0$ we can still relate the action value of a critical point to its period (see \cite[Lemma~3.3.]{cieliebak2009a}). We summarize our findings in the following lemma:
	\begin{lemma}
		Let $H^\chi_t(x):=\chi(t)\cdot H_t(x)$ where $\chi$ satisfies $\int\lim\limits_{0}^1\chi(t)\D t=1$ and the flow of $H$ leaves the one form $\lambda$ with $\D\lambda=\omega$ invariant. If $(x,\tau)\in\mathrm{crit}\left(\mathcal{A}_{H_t}\right)$ we have
		\begin{equation}
			\mathcal{A}_{H^\chi_t}(x,\tau)=\tau\cdot \hat{c}_0
			\label{action-period}
		\end{equation}
		\label{Lemma:action-period}
	\end{lemma}
	Note that the addition of $\chi(t)$ does not influence the proof we discussed above.
	\begin{example}
		\label{example:action-period}
		Given $\mathbb{R}^{2n}$ with its canonical symplectic form. On it consider the Hamiltonian
		\[
		H_t(q,p)=\frac{\|p\|^2}{2}+V(t)\frac{\|q\|^2}{2}-c.
		\]
		The corresponding Hamiltonian vector field is
		\[
		X_{H_t}(q,p)=\begin{pmatrix}p\\-V(t)q\end{pmatrix}
		\]
		and as one form we take
		\[
		\lambda=\frac{1}{2}\left(p\D q-q\D p\right).
		\]
		We calculate that
		\begin{align*}
			\Lie_{X_{H_t}}\lambda=\iota_{X_{H_t}}\D\lambda+\D\left(\iota_{X_{H_t}}\lambda\right)=-\D H_t +\D\left(\frac{\|p\|^2}{2}+V(t)\frac{\|q\|^2}{2}\right)=0,
		\end{align*}
		i.e. $\lambda$ is invariant under the Hamiltonian flow. Alternatively, we can directly compute:
		\[
		\iota_{X_{H_t}}\lambda= \frac{\|p\|^2}{2}+V(t)\frac{\|q\|^2}{2}= H(q,p)+c
		\]
	\end{example}
	
	Different form the time independent case we now first have to establish the $L_\infty$ bound on the gradient flow lines to be able to prove a proposition equivalent to \cite[Proposition~3.2.]{cieliebak2009a}. As usually the uniform bound for gradient flow lines is established via the maximum principle for (sub-)harmonic functions:
	
	From now on assume that the Hamiltonian vector field $X_{H_t}$ has support in the compact set $K$. Hence, on $\mathbb{R}^{2n}\setminus K$ the gradient flow equation of $\mathcal{A}_{H_t}$ becomes
	\[
	\del_s u=-J\del_t u
	\]
	for an almost complex structure $J$. Now consider the projection of the gradient flow line $u$ to the radial direction
	\[
	\pi_r(u(s,t))=\langle e_r,u(s,t)\rangle,
	\]
	where $e_r$ is the unit radial basis vector for some appropriately chosen hyperspherical coordinates on $\mathbb{R}^{2n}$. Then:
	\begin{align*}
		\del_s^2\langle e_r,u(s,t)\rangle 
		=& \langle e_r,\del_s^2u(s,t)\rangle\\
		=& \langle e_r,-\del_s J\del_t u(s,t)\rangle\\
		=& \langle e_r,-J\del_t\del_s u(s,t)\rangle\\
		=& \langle e_r,J^2\del_t^2 u(s,t)\rangle\\
		=& -\del_t^2 \langle e_r, u(s,t)\rangle
	\end{align*}
	Hence,
	\[
	\left(\del_s^2+\del_t^2\right) \langle e_r, u(s,t)\rangle=0,
	\]
	i.e. $\pi_r(u(s,t))$ is a harmonic function. By the maximum principle this implies that no gradient flow line can leave the closed ball $B_{R_K}\supset K$. This proves the uniform bound on the gradient flow lines. 
	
	Now we turn our attention to the Lagrange multiplier: Our strategy to prove the existence of a uniform bound for the Lagrange multiplier still follows closely the original paper on RFH \cite{cieliebak2009a}. Hence, the next step is to find a proof of the following proposition, that is applicable in our slightly different setting.
	\begin{proposition}
		Let $H_t^\chi$ be compactly supported in the compact set $K\subset \mathbb{R}^{2n}$ with all the assumptions from Lemma~\ref{Lemma:action-period} with the additional requirement that the constant $\hat{c}_0$ from equation \eqref{action-period} is not zero. Then there exists for every $\epsilon>0$ an $\alpha <\infty$ such that the following implication holds:
		\[
		\|\nabla \mathcal{A}_{H^\chi_t}(v,\eta)\|\le \epsilon \quad \Rightarrow\quad |\eta| \le \alpha\left(\left| \mathcal{A}_{H^\chi_t}(v,\eta)+1\right|\right)
		\]
		for $(v,\eta)$ part of any gradient flow line, i.e. for a gradient flow line $(u,\xi)$ there exists an $s_0\in\mathbb{R}$ such that $(v,\eta)=(u(s_0,\ \cdot\ ),\xi(s_0))$.
		\label{etaestimate}
	\end{proposition} 
	\begin{proof}
		Given an $\epsilon$ and assume that $\|\nabla \mathcal{A}_{H^\chi_t}(v,\eta)\|\le \epsilon$ for a pair $(v,\eta)$ that is part of a gradient flow line. This implies that 
		\begin{equation}
			\|\del_t v-\eta X_{H_t^\chi}(v)\|_{L^2}\le \epsilon
			\label{smallgradient}
		\end{equation}
		and 
		\[
		\left| \int\limits_{S^1} {H_t^\chi}(v(t))\text{d}t\right|\le \epsilon
		\]
		by the definition of the gradient of $\mathcal{A}_{H^\chi_t}$. We can write
		\begin{align*}
			\left|\mathcal{A}_{H^\chi_t}(v,\eta)\right|=& \left|\int\limits_{S^1}\lambda(v)\del_tv\mathrm{d}t	-\eta \int\limits_{S^1} {H_t^\chi}(v(t))\text{d}t\right| \\
			=& \left|\eta\int\limits_{S^1}\lambda(v)X_{H^\chi_t}(v)\mathrm{d}t+ \int\limits_{S^1}\lambda(v)\left(\del_tv-\eta X_{H^\chi_t}(v)\right)\mathrm{d}t -	\eta \int\limits_{S^1} {H_t^\chi}(v(t))\text{d}t\right|.
		\end{align*}
		Doing the same calculations as for the proof of Lemma~\ref{Lemma:action-period} for the first component we can get
		\begin{align*}
			\left|\mathcal{A}_{H^\chi_t}(v,\eta)\right| 
			=&\left|\eta\int\limits_{S^1}{H^\chi_t}(v(t))+c_0(t)\mathrm{d}t+ \int\limits_{S^1}\lambda(v)\left(\del_tv-\eta X_{H^\chi_t}(v)\right)\mathrm{d}t -	\eta \int\limits_{S^1} {H_t^\chi}(v(t))\text{d}t\right|\\
			=&\left|\eta\int\limits_{S^1}c_0(t)\mathrm{d}t+ \int\limits_{S^1}\lambda(v)\left(\del_tv-\eta X_{H^\chi_t}(v)\right)\mathrm{d}t\right|
		\end{align*}
		Using the reverse triangle inequality and the H\"older inequality we can estimate:
		\begin{align*}
			\left| \mathcal{A}_{H^\chi_t}(v,\eta)\right|\ge & |\eta|\cdot |c_0|-\max\limits_{x\in K}|\lambda(x)|\cdot \left\|\del_tv-\eta X_{H^\chi_t}(v)\right\|_{L^1}\\
			\ge& |\eta|\cdot |c_0|-\max\limits_{x\in K}|\lambda(x)|\cdot \left\|\del_tv-\eta X_{H^\chi_t}(v)\right\|_{L^2}
		\end{align*}
		The inequality from the beginning of the proof \eqref{smallgradient} then allows us to further estimate 
		\[
		\left| \mathcal{A}_{H^\chi_t}(v,\eta)\right|\ge |\eta|\cdot |c_0|- c_\lambda\cdot\epsilon,
		\]
		where we denote the maximum of $|\lambda(x)|$  on $K$ by the constant $c_\lambda$. If the constant $c_0$ is not zero, we can rearrange this inequality to
		\[
		|\eta|\le \frac{1}{|c_0|}\left|\mathcal{A}_{H^\chi_t}(v,\eta)\right|
		+\frac{c_\lambda\cdot \epsilon}{|c_0|}.
		\]
		By choosing
		\[
		\alpha:=\max\left\{\frac{1}{|c_0|}, \frac{c_\lambda\cdot \epsilon}{|c_0|}\right\}
		\]
		we prove the claim of the proposition:
		\[	
		|\eta|\le \alpha\cdot \left(\left|\mathcal{A}_{H^\chi_t}(v,\eta)\right|+1\right)
		\]
	\end{proof}
	Note that the constants $c_0$ and $c_\lambda$ used in the proof above change continuously along a given family $H^\sigma_t$ $\sigma\in [0,1]$ of Hamiltonians satisfying the condition of Proposition~ \ref{etaestimate}. So by taking the maximum over the entire family of constants
	\[
	\alpha:=\max\limits_{\sigma\in [0,1]}\left\{\frac{1}{|c_0|}, \frac{c_\lambda\cdot \epsilon}{|c_0|}\right\}
	\]
	we can extent the implication
	\[
	\|\nabla \mathcal{A}_{H^{\sigma,\chi}_t}(v,\eta)\|\le \epsilon \quad \Rightarrow\quad |\eta| \le \alpha\left(\left| \mathcal{A}_{H^{\sigma,\chi}_t}(v,\eta)+1\right|\right)
	\]
	to the whole family with a uniform $\alpha$. 
	
	With this implication at our hand we can now prove the uniform bound for the Lagrange multiplier (c.f. \cite[Corollary~3.5]{cieliebak2009a}).
	\begin{corollary}
		Assume that $w:=(u,\xi)\in C^\infty\left(\mathbb{R}\times S^1,\mathbb{R}^{2n}\right)$ is a (negative) gradient flow line of $\mathcal{A}_{H^\chi_t}$, where we additionally require
		\[
		a\le \mathcal{A}_{H^\chi_t}(w(s))\le b \qquad \forall s\in\mathbb{R}
		\]
		for $a,b\in\mathbb{R}$. If further $H_t^\chi$ satisfies all the requirements of Proposition~\ref{etaestimate}, then the $L_\infty$-norm of $\eta$ is uniformly bounded in terms of a constant which only depends on $a$ and $b$.
	\end{corollary}
	\begin{proof}
		Fix an $\epsilon>0$. For $s\in\mathbb{R}$ let $\sigma(s)\ge 0$ be defined by 
		\[
		\sigma(s):=\inf\left\{\sigma\ge 0\ :\ \left\|\nabla \mathcal{A}_{H^{\chi}_t}(w(s+\sigma))\right\|<\epsilon\right\}
		\]
		Claim:
		\begin{equation}
			\sigma(s)\le \frac{b-a}{\epsilon^2} \qquad \forall s\in \mathbb{R}
			\label{sigmaestimate}
		\end{equation}
		To see this we estimate
		\begin{align*}
			b-a\ge& \lim\limits_{s\to-\infty}\mathcal{A}_{H^\chi_t}(w(s))-\lim\limits_{s\to\infty}\mathcal{A}_{H^\chi_t}(w(s))\\
			=& \int\limits_{-\infty}^\infty \left\|\nabla \mathcal{A}_{H^{\chi}_t}(w(s))\right\|^2\D s\\
			\ge& \int\limits_{s}^{s+\sigma(s)} \left\|\nabla \mathcal{A}_{H^{\chi}_t}(w(s))\right\|^2\D s,
		\end{align*}
		for any choice of $s\in\mathbb{R}$. By the definition of $\sigma(s)$ the norm of the gradient $\left\|\nabla\mathcal{A}_{H^{\chi}_t}(w(s))\right\|$ is bigger equal $\epsilon$ on the open interval $(s,s+\sigma(s))$. Hence, we can conclude 
		\[
		b-a\ge \sigma(s)\cdot\epsilon^2 \qquad \forall s\in\mathbb{R}.
		\]
		This proves the claim. 
		
		Next, define $M:=\max\left\{|a|,|b|\right\}$. Utilizing Proposition~\ref{etaestimate} and the definition of $\sigma(s)$ we can prove the existence of a constant $c_M$ such that 
		\begin{equation}
			\begin{aligned}
				\left|\xi(s+\sigma(s))\right| \le& \alpha\left(\left| \mathcal{A}_{H^\chi_t}(u,\xi)+1\right|\right)\\
				\le& \alpha(M+1)=:c_M.
			\end{aligned}
			\label{c_M}
		\end{equation}
		Note that since $\alpha$ is independent of the specific gradient flow line, so is $c_M$. Equation \eqref{sigmaestimate} and \eqref{c_M} allow us now to prove the claim of the corollary:
		\begin{align*}
			|\xi(s)|\le& |\xi(s+\xi(s))|+\int\limits_{s}^{s+\sigma} \left|\del_s\xi(\hat{s})\right|\D\hat{s}\\
			=&|\xi(s+\sigma(s))|+ \int\limits_{s}^{s+\sigma} \left|\del_s\xi(\hat{s})\right|\D\hat{s}\\
			\le& c_M+\left\|H_t^\chi\right\|_\infty\cdot\sigma(s)\\
			\le& c_M+\frac{\left\|H_t^\chi\right\|_\infty\cdot(b-a)}{\epsilon^2}
		\end{align*}
		Note that the second row follows from the fact that $(u,\xi)$ is a gradient flow line. Since the right hand side is independent of $s$ we get 
		\[
		\|\xi\|_\infty\le c_M+\frac{\left\|H_t^\chi\right\|_\infty\cdot(b-a)}{\epsilon^2}
		\]
	\end{proof}
	
	The last part of proving the compactness of the moduli space of gradient flow lines is to show that the derivatives of the loop $u(s)$ are bounded. But since we work on $\mathbb{R}^{2n}$ the symplectic form is always exact and therefore we can rely on the usual proof via the non existence of nonconstant $J$-holomorphic spheres.
	
	What we proved so far shows that in our time dependent case we have a well-defined RFH. In the next step we want to show that this new type of RFH is also invariant under the homotopy of $H$: So given a homotopy 
	\[
	H^\sigma_t \quad \sigma\in [0,1].
	\]
	The typical way construct to a continuation isomorphism between the corresponding RFH's is by considering gradient flow lines where the corresponding Hamiltonian changes along the flow following the given homotopy, i.e.
	\[
	\del w(s)=\nabla \mathcal{A}_{H^{\chi,s}_t}(w(s)).
	\]
	Note that in the previous discussion we established that for any $\epsilon>0$ the implication
	\[
	\|\nabla \mathcal{A}_{H^{\sigma,\chi}_t}(v,\eta)\|\le \epsilon \quad \Rightarrow\quad |\eta| \le \alpha\left(\left| \mathcal{A}_{H^{\sigma,\chi}_t}(v,\eta)+1\right|\right)
	\]
	where $\alpha$ is independent of $\sigma\in[0,1]$. Therefore one can follow the arguments in \cite{cieliebak2009a} (see especially chapter~3.2 and Theorem~3.6) to show that the continuation isomorphism is well defined.
	
	The last part in our discussion of the well-definedness of RFH in the time dependent case concerns the critical points. Do define RFH one needs the critical points to be isolated (Morse case) or at least to form closed finite dimensional manifolds (Morse Bott case). For the non constant orbits (i.e. $\tau\ne0$) we can use the standard argument from Floer homology that shows that for a generic Hamiltonian these orbits are isolated critical points. Note that in contrast to time independent Hamiltonians we don't have an $S^1$-invariant functional, hence, we can achieve isolated critical points after perturbation. In the case that $\del_t H_t=0$ we can of course fall back to \cite[Appendix~B]{cieliebak2009a}. For the constant orbits ($\tau=0$) we need to give a new argument why they form a closed submanifold, since in the time dependent case they can not be characterised as $H^{-1}(0)$ any more. But instead for
	\[
	F_{H_t}:\mathbb{R}^{2n} \to \mathbb{R}, \quad z\mapsto \int\limits_{S^1}H_t(z)\mathrm{d}t
	\]
	the set of constant critical points is given by $F^{-1}_{H_t}(0)$. The next example will show that this set is a closed manifold for the Hamiltonians we are interested in.
	\begin{example}
		\label{example:ismanifold}
		Consider the Hamiltonian
		\[
		H_t(q,p)=\frac{\|p\|^2}{2}+(E-V(t))\cdot \frac{\|q\|^2}{2}-c,
		\]
		for a $c>0$. Then a point $(q_0,p_0)$ in $ F_{H_t}^{-1}(0)$ has to satisfy
		\[
		\frac{\|p\|^2}{2}+\frac{\|q\|^2}{2}\int\limits_{S^1}(E-V(t))\D t-c=0.
		\]
		The point $(q_0,p_0)$ is a regular point if in addition
		\[
		\nabla F_{H_t}(q_0,p_0)=\left(q_0\cdot\int\limits_{S^1}(E-V(t))\D t,p_0\right)\ne (0,0).
		\]
		The following two cases can appear:
		\begin{enumerate}
			\item If $(q_0,p_0)\in\mathbb{R}^n\times\{0\}$, then 
			\[
			\frac{\|q_0\|^2}{2}\cdot\int\limits_{S^1}(E-V(t))(t)\D t-c=0.
			\]
			Since $c>0$ this implies that both $q_0$ and $\int\limits_{S^1}(E-V(t))\D t$ are non zero and therefore the gradient $\nabla F_{H_t}(q_0,p_0)$ does not vanish.
			\item If $(q_0,p_0)\notin\mathbb{R}^n\times\{0\}$, then we see immediately that $\nabla F_{H_t}(q_0,p_0)\ne(0,0)$.
		\end{enumerate}
		Therefore all elements in $F_{H_t}^{-1}(0)$ are regular points.
	
	 For a well defined RFH we further need this set to be a closed manifold. Unfortunately this is not always the case (e.g. if $(E-V(t))=-1$ for all $t$). But by requiring that 
	\[
	\int\limits_{S^1}(E-V(t))\D t>0
	\] 
	we can ensure that the manifold of critical points with $\tau=0$ is a closed manifold.
	\end{example}
	\subsection{The Case of Lagrangian Boundary Conditions}
	\label{Lagrangiancase}
	Instead of periodic Hamiltonians and periodic orbits, one can similarly consider Hamiltonians of the form 
	\[
	H:[0,1]\times\mathbb{R}^{2n} \to \mathbb{R}
	\]
	and chords starting and ending in a given Lagrangian $L\subset \mathbb{R}^{2n}$. The corresponding Rabinowitz Floer action functional is then
	\[
	\mathcal{A}_{H_t}:=\int\limits_{0}^1 x^*\lambda -\tau\int\limits_{0}^1H_t(x(t))\D t
	\]
	analogue to the one defined in Section~\ref{periodiccase}. Since $[0,1]$ is still a compact set Lemma~\ref{action-period}, Proposition~\ref{etaestimate} and all subsequent conclusions still hold true. This allows us to refer to \cite{merry2014a} for the well definedness of the Lagrangian version of time dependent RFH.  Note that here the critical points $(x,\tau)$ of $\mathcal{A}_{H_t}$ correspond to Hamiltonian chords
	\[
	\frac{\D}{\D t}x(t)=X_{H_{\frac{t}{\tau}}}(x(t))
	\]
	such that $x(0)$ and $x(\tau)$ are inside the chosen Lagrangian submanifold.

	\section{Applications to Quantum Mechanics on one Dimensional Configuration Spaces}
	The question we set out to answer is: For a given potential $\tilde{V}$ and any given energy $E$ can we always find a radius $R$ such that we can realize an energy eigenstate with energy $E$ on a ring of radius $R$ (alt. a box of length $R$)?
	
	\subsection{Periodic RFH and the 'Particle on a Ring'}
	The goal of this section is to use the time-dependent RFH we constructed in Section~\ref{periodiccase} to give a positive answer to the above question for energy eigenstates confined to a ring. First consider the Hamiltonian
	\[
	H:S^1\times \mathbb{R}^4 \to \mathbb{R},\quad (q,p)\mapsto H_t(q,p)=\frac{\|p\|^2}{2}+(E-V(t))\frac{\|q\|^2}{2}-c,
	\]
	where $V(t)$ is the restriction of the potential in the plain to $S^1$ and $E,c\in\mathbb{R}^+$ are constants. In Example~\ref{example:action-period} and Example~\ref{example:ismanifold} we already showed that this Hamiltonian satisfies almost all the conditions needed for a well-defined (time dependent) RFH. However, it has obviously no compactly supported vector field $X_{H_t}$. One way to solve this problem is by multiplying the Hamiltonian with a suitable cut off function. But this might interfere with the Hamiltonian trajectories that cross into the cut-off region. To avoid this we first show that any potential trajectory of $H_t$ with CZ-index small then a given constant $\eta_0$ has to ly in a fixed compact set $K_{\eta_0}$ and then choose the cut-off function to be constant one on $K_{\eta_0}$.
	
	If we assume that $(E-V(t))>0$ for all $t\in S^1$, then the condition 
	\[
	\int\limits_{S^1}H_t(x(t))\D t=0
	\]
	that has to be satisfied by every critical point of $\mathcal{A}_{H_t}$ implies that these trajectories can not leave the compact set
	\[
	\mathbb{R}^4\supset\Omega:=\left\{z\in\mathbb{R}^4\ :\ H_t(x)\le 0\ \text{for a } t\in S^1\right\}
	\]
	for too long. In particular, we can assume that a given periodic orbit $x$ starts at $t=0$ inside $\Omega$. Since it satisfies the Hamiltonian equation 
	\[
	\frac{\D}{\D t}x(t)=\tau X_{H_t}(x(t))=\tau \begin{pmatrix} 0&\id \\-(E-V(t))\id&0\end{pmatrix}x(t)
	\]
	the maximum speed at which the orbit can move away from $\Omega$ is determined by $\tau$, $E-V(t)$ and $\|x(t)\|$. So for a given $E$ and $V$ as well as a fixed $\tau_0$ we can estimate the minimal amount of time the orbit will need to go from the inside of $B_{R_i}(0)$ to  the complement of $B_{R_j}(0)$ for $R_j>R_i$ (c.f. Figure~\ref{speedlimit}). Now we add up these minimal times as we successively increase the radii. Since the speed of the orbit only increases linearly with distance to the origin there will be a radius $R_\mathrm{max}^{\tau_0}$ such that the orbit would need more time then $[0,1]$ to leave the corresponding ball. Hence, critical points $(x,\tau)$ with $\tau \leq \tau_0$ can not leave $B_{R_\mathrm{max}^{\tau_0}}(0)$.
	
	\begin{figure}
		\begin{center}
			\begin{tikzpicture}
				\draw[->] (-3,0)--(3,0);
				\draw[->] (0,-3)--(0,3);

				\draw[blue!70!white] (0,-0.7) to[out=160, in=210, looseness=2] (0,0.7);
				\draw[blue!70!white] (0.4,0.4) node {$\Omega$};
				\draw[blue!70!white] (0,0.7) to[out=30, in=340, looseness=2] (0,-0.7);
				\draw[green!50!black] (0,0) circle[radius=1cm];
				\draw[green!50!black] (1.,-1.1) node {$B_{R_1}(0)$};
				\draw[orange] (0,0) circle[radius=2cm];
				\draw[orange] (1.7,-1.8) node {$B_{R_2}(0)$};
				\draw[purple!50!black] (0,0) circle[radius=3cm];
				\draw[purple!50!black] (2.5,-2.6) node {$B_{R_3}(0)$};
				\draw(-0.2,0.2)--(-0.71,0.71);
				\draw[|-|] (-0.71,0.71)--(-1.414,1.414);
				\draw[|-|] (-1.414,1.414)--(-2.12,2.12);
				\draw[->](-2.12,2.12)--(-2.5,2.5);
				\draw (-0.8,1.2) node {$\Delta t_1$};
				\draw (-1.5,1.9) node {$\Delta t_2$};
				\draw (-0.2,0.2) node {$\bullet$};
			\end{tikzpicture}
		\end{center}
		\caption{Estimating the minimal time a Hamiltonian trajectory needs to go from the inside of $B_{R_i}(0)$ to the complement of $B_{R_j}(0)$.}
		\label{speedlimit}
	\end{figure}
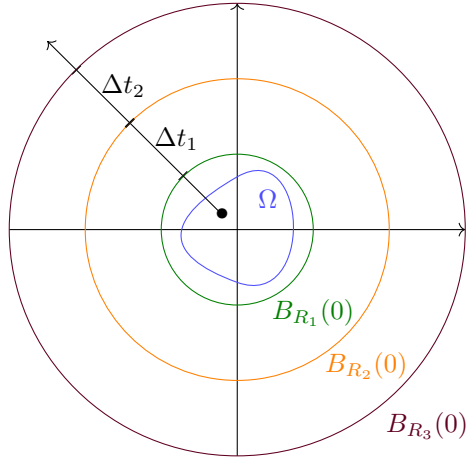
	
	Note that the Hessian of $H_t$ is a positive definite matrix at every point and its corresponding equation of motion decouples. This implies that the Conley Zehnder index (CZ-index) is strictly monotone increasing in the sense that the associated winding number is strictly monotone increasing with the period of the orbit. For more details we refer to Appendix~\ref{otto}. 
	
	Combining this with the fact that critical points $(x,\tau)$ with $\tau \leq \tau_0$ can not leave $B_{R_\mathrm{max}^{\tau_0}}(0)$ we can conclude that for a given bound on the CZ-index $\eta_0$ there exits a radius $R_{\eta_0}$ such that all orbits with CZ-index less then $\eta_0$ can't leave $B_{R_{\eta_0}}(0)$. With this we define the compactly supported Hamiltonian
	\begin{equation}
		\hat{H}^{\eta_0}_t(q,p):=\chi_{R_{\eta_0}}(q,p)\cdot\left(\frac{\|p\|^2}{2}+(E-V(t))\frac{\|q\|^2}{2}-c\right)+\left(1-\chi_{R_{\eta_0}}\right),
		\label{cutoff}
	\end{equation}
	where $\chi_{R_{\eta_0}}$ is a smooth bump function for the set $\overline{B_{R_{\eta_0}}(0)}$. Then the critical points of $\mathcal{A}_{\hat{H}^{\eta_0}_t}$ coincide with the ones of $\mathcal{A}_{H_t}$ as long as they have CZ-index smaller the $\eta_0$. 
	
	The Hamiltonian $\hat{H}^{\eta_0}_t$ fulfils all the requirements for a well defined time dependent RFH. The next step is now to compute the RFH for this Hamiltonian: First we define the homotopy
	\[
	F_s(q,p,t):=\chi_{R_{\eta_0}}(q,p)\cdot\left(\frac{\|p\|^2}{2}+(E-(1-s)V(t))\frac{\|q\|^2}{2}-c\right)+\left(1-\chi_{R_{\eta_0}}\right),
	\]
	for $s\in[0,1]$. This homotopy fulfils all the requirements for a well-defined time-dependent RFH (after a potential small perturbation) with $F_{s=0}=\hat{H}^{\eta_0}_t$ and for $s=1$ the Hamiltonian is that of a compactly supported harmonic oscillator. Note that we can always choose $R_{\eta_0}$ a bit bigger such that $F_1^{-1}(0)$ is inside the ball of radius $R_{\eta_0}$. For $F_1$ the above defined RFH for non-autonomous Hamiltonians coincides with the usual (time independent)  RFH, which is well known to be zero. By homotopy invariance we then conclude
	\[
	RFH\left(\mathbb{R}^4,\hat{H}^{\eta_0}_t\right)=0.
	\]
	Further, the constant orbits $z\in\mathbb{R}^4$ with
	\[
	\int\limits_{S^1} \hat{H}^{\eta_0}_t(z)\D t=0
	\]
	are always critical points of $\mathcal{A}_{\hat{H}^{\eta_0}_t}$ and form a three dimensional submanifold of $\mathbb{R}^4$. Assume that there exists no orbit with CZ-index smaller equal $4$, then $RFH_i\left(\mathbb{R}^4,\hat{H}^{\eta_0}_t\right)$ is well-defined for $i\in\{1,2,3,4\}$ and coincides with the singular homology of the submanifold of constant critical points. This contradiction implies that there has to exists at least on orbit with CZ-index $4$. Since the above discussion does not depend on a specific choice of $\eta_0$ we are free to choose $\eta_0=5$ in this case. This proves that for every $E-V(t)$ there exists at least one periodic solution of the Hamiltonian
	\[
	\frac{\|p\|^2}{2}+\left(E-V\left(\frac{t}{\tau}\right)\right)\frac{\|q\|^2}{2}\
	\]
	with $E-V(t)>0$ for all $t\in S^1$ and period $\tau\in\mathbb{R}^+$. 
	
	The next step is now to translate this result back into our original question about energy eigenstates on closed one dimensional manifolds. So given a trajectory $x(t)$ of the above Hamiltonian with period $\tau$. Then \[
	\psi(t):=\pr_q(x(t))
	\]
	interpreted as a number in $\mathbb{C}$ is an energy eigenstate for any given energy $E$ on a closed string $\Gamma$ of length $\tau$ and an exterior potential $\tilde{V}:\mathbb{R}^2\to\mathbb{R}$, such that
	\[
	\tilde{V}\bigg\vert_\Gamma=V\left(\frac{\argument}{\tau}\right).
	\] 
	This observation allows us now to prove Theorem~\ref{thmclosed}:
	\begin{proof}[of Theorem~\ref{thmclosed}]
		Given a radially invariant potential $\tilde{V}$ and an energy $E>\max\tilde{V}$ the Hamiltonian
		\[
		H:S^1\times \mathbb{R}^4 \to \mathbb{R},\quad (q,p)\mapsto H_t(q,p)=\frac{\|p\|^2}{2}+(E-V(t))\frac{\|q\|^2}{2}-c,
		\]
		where $V(t)$ is the restriction of $\tilde{V}$ to the unit circle in the plain, satisfies all the requirements we imposed in the above discussion. Hence, we know that there exists a periodic orbit $x(t)$ with period $\tau$ of the Hamiltonian 
		\[
		\frac{\|p\|^2}{2}+\left(E-V\left(\frac{t}{\tau}\right)\right)\frac{\|q\|^2}{2}.
		\]
		By the radial invariance of $\tilde{V}$ we know that $V\left(\frac{t}{\tau}\right)$ is equal to the restriction of $\tilde{V}$ to the circle of radius $\tau$. Renaming $\tau$ to $r_E$ to indicate the dependence on the energy we consider and taking
		\[
		\psi(t):=\pr_q(x(t))
		\]
		as the energy eigenstate on $\del B_{r_E}(0)$ yields the statement of the theorem.
	\end{proof}
	
	\subsection{Lagrangian RFH and the 'Particle in a Box'}
	\label{openstring}
	The setup is the same as in the previous section, but now we consider energy eigenstates that are confined to a straight line in $\mathbb{R}^2$. The Hamiltonian is given as  
	\[
	H:[0,1]\times \mathbb{R}^4 \to \mathbb{R},\quad (q,p)\mapsto H_t(q,p)=\frac{\|p\|^2}{2}+(E-V(t))\frac{\|q\|^2}{2}-c.
	\]
	If $E-V(t)>0$ for all $t\in[0,1]$ we can define the non-autonomous Lagrangian RFH for the modified Hamiltonian $\hat{H}_t^\eta$ (analogue to \eqref{cutoff}) and the Lagrangian $\{0\}\times\mathbb{R}^2\subset \mathbb{R}^4$. Using the same homotopy as in the previous section we again see that
	\[
	\mathrm{RFH}\left(\mathbb{R}^4,\hat{H}_t^\eta,\{0\}\times\mathbb{R}^2\right)=0,
	\]
	For the Lagrangian case the manifold of constant critical points is $F^{-1}(0)\cap \{0\}\times \mathbb{R}^2$ and, hence, one dimensional. By our convention the maximum of this manifold has index one. Again arguing by contradiction this implies that there has to exist a non constant Hamiltonian chord of $\hat{H}_t^\eta$ with index $2$ that starts and ends in $\{0\}\times\mathbb{R}^2$. To make sure that this chord is also a solution for $H_t$ we argue as follows: 
	
	Given a chord $x$ of $H_t$ that starts and ends in $\{0\}\times\mathbb{R}^2$. Then there is an anti-symplectic map 
	\[
		\zeta:\mathbb{R}^4 \to \mathbb{R}^4; \qquad (q,p)\mapsto (-q,p)
	\]
	such that that $\{0\}\times\mathbb{R}^2$ is its fixed point set and $H_t$ is invariant under its pullback. This implies that the chord
	\[
		\check{x}: [0,1]\to \mathbb{R}^4; \qquad t\mapsto \zeta(x(1-t))
	\]
	is also a Hamiltonian trajectory of $H_t$ and the concatenation $\check{x}\circ x$ is a periodic orbit of $H_t$. Note that for Lagrangian RFH the index of a trajectory is given by the Robbin-Salamon index (RS-index) for the Lagrangian path defined by applying the linearized Hamiltonian flow to the Lagrangian our chords starts and ends in. However, using the uniqueness property of the CZ-index and the concatenation property of the RS-index we can get the following relation:
	\[
		\mu_\mathrm{CZ}(\check{x}\circ x)=\mu_\mathrm{RS}(\check{x}\circ x)= 2\mu_\mathrm{RS}(x)
	\]
	So if we choose $\eta=4$ like in the previous section we know for $x$ being the index two orbit of $\hat{H}_t^\eta$ from above that $\check{x}\circ x$ can not reach the area around the cut-off and therefore is $x$ also a trajectory for $H_t$. We conclude that for every $E\in\mathbb{R}^+$ and every $V:[0,1]\to\mathbb{R}$ such that $E-V(t)>0$ for all $t\in [0,1]$ there exists a Hamiltonian trajectory $x(t)$ of the Hamiltonian
	\[
	\frac{\|p\|^2}{2}+\left(E-V\left(\frac{t}{\tau}\right)\right) \frac{\|q\|^2}{2}\
	\]
	that satisfies $x(0)$, $x(\tau)\in\{0\}\times\mathbb{R}^2$ for a $\tau\in\mathbb{R}^+$. This allows us to prove Theorem~\ref{thmopen}:
	
	\begin{proof}[of Theorem~\ref{thmopen}]
		Given a radial invariant potential $\tilde{V}:\mathbb{R}^2\to\mathbb{R}$, then for every energy $E$ with 
		\[
		E>\max\tilde{V},
		\]
		there is a solution $x(t)$ to the Hamiltonian equation for  
		\[
		\frac{\|p\|^2}{2}+\left(E-V\left(\frac{t}{\tau}\right)\right) \frac{\|q\|^2}{2}\
		\]
		which comes back to the Lagrangian submanifold $\{0\}\times\mathbb{R}^2$ at time $\tau$, where $V$ is defined to be the restriction of $\tilde{V}$ to $\hat{\Gamma}:=\{(s,1)\ :\ s\in[0,1]\}$. Again we rename $\tau$ to $l_E$ to indicate the dependence on the chosen energy. To finish the proof choose a straight line $\Gamma_{l_E}$ of length $l_E$ that starts in $(0,l_E)$ and ends in $(l_E,l_E)$. The radial invariance then implies that
		\[
		V\left(\frac{\argument}{l_E}\right)=\tilde{V}\big\vert_{\Gamma_{l_E}}.
		\]
		\begin{figure}
			\begin{center}
				\begin{tikzpicture}
					\draw[->] (-0.5,0)--(2.5,0);
					\draw[->] (0,-0.5)--(0,2.5);
					\draw[blue] (0,1)--(1,1);
					\draw[blue] (0.5,1.3) node {$\hat{\Gamma}$};
					\draw[dashed] (0,0)--(2.5,2.5);
					\draw[green!60!black] (0,2)--(2,2);
					\draw[green!60!black] (1,2.3) node {$\Gamma_{l_E}$};
					\draw (2.3,-0.3) node {$x$};
					\draw (-0.3,2.3) node {$y$};
				\end{tikzpicture}
			\end{center}
			\label{line}
			\caption{Sketch of the lines $\hat{\Gamma}$ and $\Gamma_{l_E}$ and their relation to each other.}
		\end{figure}
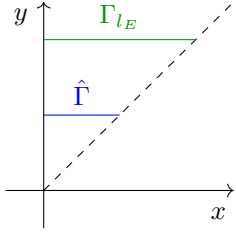
		The projection of the solution $x(t)$ onto the position space $\mathbb{R}^2$ after identifying $\mathbb{R}^2$ with $\mathbb{C}$ is an energy eigenstate confined to $\Gamma_{l_E}$.
	\end{proof}

	\appendix
	\section{Hamiltonians with monotone increasing CZ-index}
	\label{otto}
	The goal of this section is to show that for certain types of Hamiltonian flows the CZ-index grows strictly monotone with the period of the corresponding orbit. What follows is based on unpublished notes of Otto van Koert.
	
	\subsection{Conley-Zehnder Index via Winding Number}
	Given a Hamiltonian trajectory $\gamma$ in $\mathbb{R}^{2n}$. To compute its CZ-index we have to consider the linearized Hamiltonian flow along the path of $\gamma$. The result is a map
	\[
	\phi:[0,\tau]\times\mathbb{R}^{2n}\to \mathbb{R}^{2n}; \quad (t,v)\mapsto \phi(t)v,
	\]
	where $\phi(t)$ is an element in the group of symplectic matrices $Sp(2n)$ for all $t\in [0,\tau]$. To define the CZ-index as a Maslov index we have to study the intersection of the path $\phi(t)$ for $t\in [0,\tau]$ with the Maslov cycle (c.f.\cite{salamon1992a})
	\[
	\Gamma:=\left\{S\in Sp(2n)\ :\ \det(S-\id)=0 \right\}.
	\]
	Note that this set is naturally dividing $Sp(2n)$ into two parts:
	\begin{align*}
		C_+:=& \left\{S\in Sp(2n)\ :\ \det(S-\id)>0 \right\}\\
		C_-:=& \left\{S\in Sp(2n)\ :\ \det(S-\id)<0 \right\}
	\end{align*}
	To get a the right CZ-index the path in $Sp(2n)$ has to have specific start- and endpoints. Since $\phi$ is the linearization of a Hamiltonian flow it naturally starts at $\id$. The endpoint we have to choose depends on the complement of the Maslov cycle in which $\phi(\tau)$ lies. For $C_+$ we define 
	\[
	B_+:=\diag(-1,\ldots,-1)
	\]
	and for $C_-$
	\[
	B_-:=\diag(2,0.5,-1,\ldots,-1).
	\]
	So if $\phi(\tau)\in C_+$ we choose a path in $C_+$ connecting $\phi(\tau)$ to $B_+$ without intersecting $\Gamma$. Similar for $\phi(\tau)\in C_-$. 
	
	To compute the intersection number of this extended path with the Maslov cycle it is common to reinterpret this number as a winding number in $S^1\subset\mathbb{C}$. So the next step is to construct a retract from the symplectic matrices $Sp(2n)$ to the unitary matrices with complex coefficients $U\left(n,\mathbb{C}\right)$. Note that there is an inclusion of a matirx $A+iB \in U\left(n,\mathbb{C}\right)$ into $Sp(2n)$ by 
	\[
	A+iB \mapsto\begin{pmatrix}
		A&B\\ -B&A
	\end{pmatrix}.
	\]
	We shall call the image of $U\left(n,\mathbb{C}\right)$ under this inclusion by $U\left(2n,\mathbb{R}\right)$. This subgroup of $Sp(2n)$ can be characterized as
	\[
	U\left(2n,\mathbb{R}\right)=\left\{A\in Sp(2n)\ :\ JA=AJ\right\},
	\] 
	where $J$ is the standard almost complex structure on $\mathbb{R}^{2n}$. By definition every element $S$ in $Sp(2n)$ satisfies the equation
	\[
	S^\intercal J S=J.
	\]
	So if in addition $S$ is an orthogonal matrix, it satisfies
	\begin{align*}
		S^\intercal J S=&J\\
		\Leftrightarrow \qquad S\circ S^\intercal J S =& S\circ J\\
		\Leftrightarrow \qquad  J S =& SJ,
	\end{align*}
	i.e. it is in $U\left(2n,\mathbb{R}\right)$. Therefore we define the following retract:
	\[
	\tilde{\varrho}: Sp(2n) \to U\left(2n,\mathbb{R}\right), \quad S\mapsto\left(S S^\intercal \right)^{-\frac{1}{2}} S
	\]
	Note that $S S^\intercal$ is a product of symplectic matrices, hence, also a symplectic matrix. Further, the square root of a symplectic positive definite matrix is again symplectic, as is the inverse of a symplectic matrix. Overall $\left(S S^\intercal \right)^{-\frac{1}{2}} S$ is again in $Sp(2n)$, i.e. the above map is well-defined. Now consider the identification
	\[
	\iota:U\left(2n,\mathbb{R}\right) \to U\left(n,\mathbb{C}\right), \quad 	 \begin{pmatrix}
		A&B\\ -B&A
	\end{pmatrix}\mapsto A+iB
	\]
	and define the retract
	\[
	\varrho:=\tilde{\varrho}\circ\iota: Sp(2n) \to U\left(n,\mathbb{C}\right).
	\]
	Observe that $\varrho\left(B_+\right)=-\id$ and $\varrho\left(B_-\right)=\id$. So if we consider the linearized flow $\phi$ extended to a path of symplectic matrices that ends in $B_+$ or $B_-$ - we shall call this extension by $\hat{\phi}$ - one should consider $\varrho\left(\hat{\phi}\right)^2$ to get a loop of matrices. 
	\begin{definition}
		Given a Hamiltonian trajectory $\gamma$ in $\mathbb{R}^{2n}$ and let $\hat{\phi}$ be the linearized Hamiltonian flow along $\gamma$ extended to a point $B_{+/-}$. The CZ-index of $\gamma$ is then defined to be the winding number of the loop
		\[
		\det\left(\varrho\left(\hat{\phi}\right)^2\right)
		\]
		in $S^1\subset\mathbb{C}$.
	\end{definition}
	
	\subsection{On the Growth of the CZ-index for Special Trajectories in $\mathbb{R}^{2n}$}
	Given a Hamiltonian system in $\mathbb{R}^2$ the linearized flow $\phi$ takes on a particular simple form since it is only a $2\times2$ matrix. This allows us to give an simple formula for the map $\varrho$ and study the behaviour of the corresponding CZ-index very explicitly.
	
	First, remember the following formula for the square root of a $2\times 2$ positive matrix $A$:
	\[
	\sqrt{A}= \frac{1}{\sqrt{\tr(A)+2\sqrt{\det(A)}}}\left(A+\sqrt{\det(A)}\id\right)
	\] 
	Note that since the linearized flow $\phi$ is a symplectic matrix, so in particular $\det(\phi)=1$, this formula simplifies  in our case to 
	\begin{equation}
		\sqrt{\phi\phi^\intercal}= \frac{1}{\sqrt{\tr\left(\phi\phi^\intercal\right)+2}}\left(\phi\phi^\intercal+\id\right).
		\label{sqrt}
	\end{equation}
	Hence, given a linearized Hamiltonian flow of the form
	\[
	\phi=\begin{pmatrix}
		a&b\\c&d\\
	\end{pmatrix}
	\]
	we calculate
	\[
	\phi\phi^\intercal=\begin{pmatrix}
		a^2+b^2 & ac+bd\\
		ac+bd & c^2+d^2
	\end{pmatrix},
	\]
	i.e. $\tr\left(\phi\phi^\intercal\right)=a^2+b^2+c^2+d^2$. Applying the formula~\eqref{sqrt} we get
	\[
	\sqrt{\phi\phi^\intercal}= \frac{1}{\sqrt{a^2+b^2+c^2+d^2+2}} \begin{pmatrix}
		a^2+b^2+1 & ac+bd\\
		ac+bd & c^2+d^2+1
	\end{pmatrix}.
	\]
	With the help of the inversion formula for $2\times 2$ matrices one can easily calculate that
	\[
	\left(\sqrt{\phi\phi^\intercal}\right)^{-1}\phi = \frac{1}{\sqrt{a^2+b^2+c^2+d^2+2}} \begin{pmatrix}
		a+d & b-c\\
		c-b & a+d
	\end{pmatrix}.
	\]
	This is now an element in $U\left(2,\mathbb{R}\right)$. To calculate $\varrho(\phi)$ we have to identify this matrix as an element in $U\left(1,\mathbb{C}\right)\cong S^1\subset\mathbb{C}$. The final result is then given as
	\[
	\varrho(\phi)=\frac{a+d+i(b-c)}{\sqrt{(a+d)^2+(b-c)^2}}=e^{i\varphi},
	\]
	where $\varphi=\arctan\left(\frac{b-c}{a+d}\right)$. The CZ-index is then given by the winding number of $\varrho(\phi)^2=e^{i2\varphi}$ in $S^1\subset\mathbb{C}$. To understand its growth behaviour we have to look at the derivative of $\varphi$ with respect to the time variable of the Hamiltonian trajectory:
	\begin{equation}
		\dot{\varphi}=\frac{(\dot{b}-\dot{c})(a+d)-(b-c)(\dot{a}-\dot{d})}{(a+d)^2+(b-c)^2}
		\label{varphidot}
	\end{equation}
	At the beginning of this section we said that our goal is to show that for certain Hamiltonians the CZ-index is strictly monotone increasing with the period of the corresponding orbit. Since the CZ-index is only integer valued this can of course never be true in a strict sense. Instead with meant that we want to prove $\dot{\varphi}>0$ along the whole orbit. 
	
	To achieve this let us first recall some facts about the linearized flow: For a given Hamiltonian function $H$ on $\mathbb{R}^{2n}$ with the canonical symplectic form the Hamiltonian flow is defined by the equation
	\[
	\frac{\D}{\D t} \Phi_t(z)=J\nabla H\left(\Phi_t(z)\right),
	\]
	where $J$ is the matrix
	\[
	J:=\begin{pmatrix}
		0& \id\\
		-\id & 0\\
	\end{pmatrix}.
	\]
	The linearized flow along an orbits $\gamma$ is then defined as the map
	\[
	\phi: \mathbb{R}\times\mathbb{R}^{2n} \to \mathbb{R}^{2n}; \quad (t,v) \mapsto \mathrm{D} \Phi_t(\gamma(0))v.
	\]
	Hence, by differentiating the defining equation for $\Phi_t$ we see that  the linearized flow satisfies the equation
	\begin{equation}
		\frac{\D}{\D t} \phi_t = JD^2H(\gamma(t)) \phi_t,
		\label{linearflow}
	\end{equation}
	where $D^2H(\gamma(t))$ denotes the Hessian of $H$ along the orbit $\gamma$. Combining this fact together with \eqref{varphidot} yields the following proposition:
	\begin{proposition}
		Given a Hamiltonian function $H$ on $\mathbb{R}^{2n}$ with the canonical symplectic form and a periodic orbit $\gamma$ with period $\tau$. If the Hessian $D^2H(\gamma(t))$ is a positive definite diagonal matrix for all $t\in[0,\tau]$, then the CZ-index of $\gamma$ is strictly monotone increasing, meaning $\dot{\varphi(t)}$ is bigger zero for all $t\in[0,\tau]$, where the CZ-index of $\gamma$ is given by the winding number of $e^{i\varphi(t)}$.
		\label{posCZ}
	\end{proposition}
	\begin{proof}
		First consider the case $n=1$: Let us abbreviate $D^2H(\gamma(t)) = D_t$.
		Then we can deduce from \eqref{linearflow}
		\begin{equation}
			\frac{\D}{\D t} \phi_t = D_t \phi_t,
			\label{ODEa1}
		\end{equation}
		Further, when we denote
		\[
		\phi_t=\begin{pmatrix}
			a&b\\
			c&d
		\end{pmatrix},
		\]
		for $a$, $b$, $c$ and $d$ functions in $t$, then we can calculate the growth of the CZ-index following equation \eqref{varphidot}. For this we write equation~\eqref{ODEa1} more explicitly as
		\[
		\begin{pmatrix}
			\dot{a}&\dot{b}\\
			\dot{c}&\dot{d}
		\end{pmatrix}= 
		\begin{pmatrix}
			D_{22}c&D_{22}d\\
			-D_{11}a&-D_{11}b
		\end{pmatrix},
		\]
		where $D_{11}$ and $D_{22}$ are the diagonal elements of $D_t$ that by assumption are positive. This allows us to rewrite the numerator of \eqref{varphidot} in the following way:
		\begin{align*}
			(\dot{b}-\dot{c})(a+d)-(b-c)(\dot{a}-\dot{d})=D_{11}\left(a^2+b^2+1\right)+D_{22}\left(c^2+d^2+1\right)>0
		\end{align*}
		This proves the claim for $n=1$.
		
		Now assume we have   
		\[
		D^2H(\gamma(t)) = D_t.
		\]
		for $n$ arbitrary. By choosing a basis of the form $(q_1,p_1,\ldots,q_n,p_n)$ equation \eqref{linearflow} can be rewritten as 
		\[
		\frac{\D}{\D t}\phi_t=\begin{pmatrix}
			J_2&&0\\
			&\ddots& \\
			0&&J_2
		\end{pmatrix}
		\begin{pmatrix}
			D_1(t)&&0\\
			&\ddots& \\
			0&&D_n(t)
		\end{pmatrix} {\phi}_t,
		\]
		where $J_2$ is the two dimensional standard almost complex structure and $D_i(t)$ is a two dimensional diagonal matrix with only positive entries. This implies that ${\phi}_t$ also has a block diagonal structure
		\[
		{\phi}_t=\begin{pmatrix}
			B_1(t)&&0\\
			&\ddots& \\
			0&&B_n(t)
		\end{pmatrix}
		\]
		with $B_i(t)$ being two dimensional matrices that fulfil the same properties as the linearized flow in the case $n=1$. Therefore, applying the retract $\varrho$ to ${\phi}_t$ results in
		\[
		\varrho\left({\phi}_t\right)=
		\begin{pmatrix}
			e^{i\varphi_1}&&0\\
			&\ddots& \\
			0&&e^{i\varphi_n}
		\end{pmatrix}
		\] 
		and finally
		\[
		\det\left(\varrho\left({\phi}_t\right)\right) = e^{i2(\varphi_1+\ldots+\varphi_n)},
		\]
		where $\dot{\varphi}_i>0$ for all $i\in\{1,\ldots,n\}$ according to the discussion for the case $n=1$. This proves the proposition.
	\end{proof}
	\begin{example}
		Consider on $\mathbb{R}^{2n}$ the time dependent harmonic oscillator
		\[
		H(q,p)=\frac{\|p\|^2}{2}+\frac{k(t)}{2}\|q\|^2.
		\]
		with $k(t)>0$ for all $t$. The corresponding Hessian is then
		\[
		D^2H=\begin{pmatrix}
			1&0&&&0\\
			0&k(t)&&&\\
			&&\ddots&&& \\
			&&&1&0\\
			0&&&0&k(t)
		\end{pmatrix}.
		\]
		From this it is easy to see that $D^2H$ is this case satisfies all the conditions from Proposition~\ref{posCZ} and therefore has strictly increasing CZ-index.
		\label{posex}
	\end{example}
	
	\bibliographystyle{alpha}
	\bibliography{Kbib}	
	
\end{document}